\documentclass[psamsfonts]{amsart}

\usepackage{amsfonts}
\usepackage{amssymb}

\newtheorem{theorem}{Theorem}[section]
\newtheorem{lemma}{Lemma}[section]
\newtheorem{prop}{Proposition}[section]
\newtheorem{cor}{Corollary}[section]
\newtheorem{assumption}[theorem]{Assumption}

\newtheorem{remark}{Remark}[section]

\newcommand{\R}{\mathbb{R}}
\newcommand{\Rd}{\mathbb{R}^d}

\newcommand{\E}{\mathbb{E}}

\newcommand{\ds}{\displaystyle}

\newcommand{\Dom}{\rm{Dom}}

\newcommand{\bean}{\begin{eqnarray*}}
\newcommand{\eean}{\end{eqnarray*}}

\newcommand{\la}{\langle}
\newcommand{\ra}{\rangle}
\newcommand{\calS}{\mathcal{S}}
\newcommand{\calE}{\mathcal{E}}

\usepackage[dvips]{graphicx}

\usepackage{graphicx}
\usepackage{amssymb}
\usepackage{amsmath}
\usepackage{color}
\usepackage{times}

\usepackage[unicode,bookmarks,colorlinks]{hyperref}
\hypersetup{
   linkcolor=brickred,
}

\definecolor{mahogany}{cmyk}{0, 0.77, 0.87, 0}
\definecolor{salmon}{cmyk}{0, 0.53, 0.38, 0}
\definecolor{melon}{cmyk}{0, 0.46, 0.50, 0}
\definecolor{yellowgreen}{cmyk}{0.44, 0, 0.74, 0}
\definecolor{brickred}{cmyk}{0, 0.89, 0.94, 0.28}
\definecolor{OliveGreen}{cmyk}{0.64, 0, 0.95, 0.40}
\definecolor{RawSienna}{cmyk}{0, 0.72, 1.0, 0.45}
\definecolor{ZurichRed}{rgb}{1, 0, 0} 

\date{}

\begin{document}


\title{Probabilistic Approach to Fractional Integrals and the Hardy-Littlewood-Sobolev Inequality}
\author{David Applebaum}\thanks{}
\address{School of Mathematics and Statistics, University of Sheffield,
Hicks Building, Hounsfield Road, Sheffield, England, S3 7RH}
\email{d.applebaum@sheffield.ac.uk}
\author{Rodrigo Ba\~nuelos}\thanks{R. Ba\~nuelos is supported in part  by NSF Grant
\# 0603701-DMS}
\address{Department of Mathematics, Purdue University, West Lafayette, IN 47907, USA}
\email{banuelos@math.purdue.edu}
\maketitle

\begin{abstract}
We give a short summary of Varopoulos' generalised  Hardy-Littlewood-Sobolev inequality for self-adjoint $C_{0}$ semigroups and give a new probabilistic  representation of the classical fractional integral operators on $\R^n$ as projections of martingale transforms.  Using this formula we derive a new proof of the classical Hardy-Littlewood-Sobolev inequality based on Burkholder-Gundy and Doob's inequalities for martingales.
\end{abstract}

\tableofcontents

\section{Introduction}  As is evident from the many recent papers on martingale transforms and their applications to singular integral operators and  Fourier multipliers on $\R^d$ (see \cite{AppBan}, \cite{Ban},  \cite{BanBau},  \cite{BanOse}, \cite{BorJanVol}, \cite{GMS}, \cite{O1}, for example),  martingale inequalities  can be very effectively used to study many operators in analysis which on the surface do not appear related to probability at all.  This point of view often leads to sharp estimates and provides new insight into the behavior of the operators. Even when the estimates are not sharp, this approach can help clarify how such bounds may depend on the geometry of the space where the operators are defined. For the latter point,  see for example \cite{BanBau} where bounds are proved for operators on manifolds with no geometric assumptions on the manifold.   In  this paper we  provide a probabilistic representation for the fractional integral operators  on $\R^d$ as  projections of martingale transforms and use this representation to give a stochastic analytic proof of the classical Hardy-Littlewood-Sobolev inequality, i.e. for the heat semigroup.   Once the representation is obatined, our proof follows from the classical Burkholder-Gundy inequalities and from Doob's inequality.  Judging from previous similar representations for singular integrals, one expects that when this representation is better understood, one would get better (and perhaps explicit) bounds for the constants given below, this time in terms of the dimension of the semigroup, which plays a crucial role on this theory.

The Hardy-Littlewood-Sobolev inequality has been extended to the general setting of $C_{0}$-semigroups by Varopoulos in \cite{Var} and these extensions have been widely studied by many researcher for several years.     In order to make this paper as self-contained  as possible and to give the non-expert a sense of the level of generality on the validity of the Hardy-Littlewood-Sobolev inequality, we review Varopoulos' general approach in \S2. The assumption that the semigroup is self-adjoint (which covers a wide range of examples that are interesting to both analysts and probabilists), enables us to simplify the proof by using Stein's maximal ergodic theorem \cite{St}.  To further illustrate with examples, we present some subordinated semigroups in \S3.  In \S4, we restrict our attention to the heat semigroup, obtain the probabilistic representation for the corresponding fractional integrals on $\R^d$, and give the probabilistic  proof of the Hardy-Littlewood-Sobolev inequality.   Such a representation and proof of the Hardy-Littlewood-Sobolev inequality, in terms of the space-time Brownian motion first studied in \cite{BanMen}, applies to manifolds with certain assumption on their geometry.   On the other hand, since it involves the gradient operator it does not apply (at least not directly) to more general semigroups. For the semigroups studied in \cite{Var1}, an alternate stochastic representation holds in terms of the construction of Gundy and Varopoulos \cite{GunVar}.  Such a representation is discussed at the end of \S4.

{\it Notation}. Let $S$ be a metric space with metric $\rho$, $g$ be a function from $S \times S$ to $(0, \infty)$ and $h$ be a function from $(0, \infty)$ to $(0, \infty)$.  Throughout this work we use the notation $g(x,y) \asymp C h\left(\frac{\rho(x,y)}{c}\right)$ to mean that there exist $C_{1}, C_{2}, c_{1}, c_{2} > 0$ so that
$$ C_{1}h\left(\frac{\rho(x,y)}{c_{1}}\right) \leq g(x,y) \leq C_{2}h\left(\frac{\rho(x,y)}{c_{2}}\right),$$ for all $x,y \in S$. Note that the values of $C_{i}$ and $c_{i} (i=1,2)$ may change from line to line.
We will denote the Schwartz space of rapidly decreasing functions on $\R^{d}$ by ${\calS}(\R^{d})$ . Note that it is dense in $L^{p}(\R^{d})$ for all $1 \leq p < \infty$.

%

\section{The Hardy-Littlewood-Sobolev Theorem and Varopoulos dimension}

\subsection{The $(n,p)$--ultracontractivity assumption}

 Let $(S, {\calS}, \mu)$ to be a measure space and let $L^{p}(S):=L^{p}(S, {\calS}, \mu; \R)$. We assume that there is a family of linear operators $(T_{t}, t \geq 0)$ which are contraction semigroups on $L^{p}(S)$ for all $1 \leq p \leq \infty$. However we only assume that the semigroup is strongly continuous in the case $p=2$.  We further assume that $T_{t}$ is self-adjoint on $L^{2}(S)$ for all $t \geq 0$.

In the proof of Theorem \ref{HLS} below, we will make use of the fact (as is shown in \cite{St}), that for all $1 < p < \infty$ there exists $D_{p} > 0$ so that for all $f \in L^{p}(S)$,
\begin{equation} \label{ass2}
 ||f^{*}||_{p} \leq D_{p}||f||_{p},
\end{equation}
where for all $x \in S, f^{*}(x) = \sup_{t > 0}|T_{t}f(x)|$. Note also that $f^{*}$ is a well-defined measurable function.

We make the following assumption, which as we shall see, is satisfied by many semigroups.

\begin{assumption}[{$(n, p)$-ultracontractivity}]\label{assumption1} There exists an $n>0$ (not required to be an integer) such that for all $1 \leq p < \infty$,  there exists $C_{p,n} > 0$  so that for all $t > 0, f \in L^{p}(S)$,
\begin{equation} \label{ass1}
||T_{t}f||_{\infty} \leq C_{p,n}t^{-\frac{n}{2p}}||f||_{p}.
\end{equation}
Following Varopoulos' terminology, the number $n$ will be referred to as the dimension of the semigroup $T_t$. \end{assumption}

Note that the semigroup $(T_{t}, t \geq 0)$ is then ultracontractive as defined, for example in \cite{Davhk}.  That is, $T_t:L^1(S) \to L^{\infty}(S)$ for all $t > 0$.  We now examine (\ref{ass1}) from the point of view of semigroups that are integral operators with positive kernels. If (\ref{ass1}) holds and we assume that the semigroup is $L^{2}$ positivity-preserving, i.e. that for all $f \in L^{2}(S)$ with $f \geq 0$ (a.e.) we have $T_{t}f \geq 0$ (a.e.) for all $t > 0$, it follows from \cite{Davhk} pp.59-60 that the semigroup has a symmetric kernel $k: (0, \infty) \times S \times S \rightarrow [0, \infty)$ so that
$$ T_{t}f(x) = \int_{S}f(y)k_{t}(x,y)\mu(dy),$$
for all $f \in L^{p}(S), x \in S, t > 0$ and moreover
$$ \sup_{x,y \in S}k_{t}(x,y) \leq c_{t}$$
where the mapping $t \rightarrow c_{t}$ is monotonic decreasing on $(0, \infty)$ with $\lim_{t \rightarrow 0}c_{t} = \infty$.

 Conversely suppose  the semigroup  $(T_{t}, t \geq 0)$ is given by a kernel so that $$T_{t}f(x) = \int_{S}f(y)k_{t}(x,y)\mu(dy)$$  for all $x \in S, f \in L^{p}(S), 1 \leq p \leq \infty$. Assume that the kernel $k \in C((0, \infty) \times S \times S)$  and is also such that
\begin{itemize}
\item $\int_{S}k_{t}(x,y)\mu(dy) = 1$ for all $t > 0, x \in S$ (so that $k_{t}(x,\cdot)$ is the density, with respect to the reference measure $\mu$, of a probability measure on $S$),
\item There exists $C > 0$ so that for all $t > 0, x,y \in S$,
$$ k_{t}(x, y) \leq C t^{-\frac{n}{2}},$$
\item $k_{t}$ is symmetric for all $t > 0$, i.e. $k_{t}(x,y) = k_{t}(y,x)$ for all $x, y \in S$.
\end{itemize}

Then (\ref{ass1}) is satisfied since by Jensen's inequality, for all $1 \leq p < \infty, x \in S, t > 0$
\bean |T_{t}f(x)|^{p} & = & \left|\int_{S}f(y)k_{t}(x,y)\mu(dy)\right|^{p}\\
& \leq & \int_{S}|f(y)|^{p}k_{t}(x,y)\mu(dy)\\
& \leq & Ct^{-\frac{n}{2}}||f||_{p}^{p}, \eean
and so
$$ ||T_{t}f||_{\infty} \leq C^{\frac{1}{p}}t^{-\frac{n}{2p}}||f||_{p}.$$

In particular, this condition is satisfied by the heat kernel on certain Riemannian manifolds where $n=d$, the dimension, and on some classes of fractals where $n = 2\frac{\alpha}{\beta}$ where $\alpha$ is the Hausdorff dimension and $\beta$ is the walk dimension (see e.g. \cite{GT}). As discussed in \S3 it holds for the $\beta$-stable transition kernel on Euclidean space and a class of Riemannian manifolds where $n = \frac{d}{\beta}$. It also holds for strictly elliptic operators on domains in Euclidean space (see \cite{Davhk} Theorem 2.3.6, pp.73-4).

\subsection{Fractional Integral Operators}

Fix $1 \leq p < \infty$ and for any $0<\alpha < n$  define a linear operator $I_{\alpha}$, called the fractional
integral of $f$,  by
\begin{equation}\label{fracint}
  I_\alpha (f)(x) = \frac{1}{\Gamma(\alpha/2)} \int_0^\infty
  t^{\alpha/2-1} T_{t}f(x)\, dt,
\end{equation}
for $f\in L^1(S)\cap L^p(S), x \in S$.

\begin{remark} We call $I_\alpha$ a fractional integral operator as it coincides with the classical Riemann-Liouville operator when $S = \R$ and $(T_{t}, t \geq 0)$ is the translation group. We may also regard it as the Mellin transform of the semigroup. 
\end{remark}

\begin{lemma} \label{frint1} The integral defining $I_\alpha(f)$ is absolutely
convergent.
\end{lemma}

\begin{proof} Fix $x \in S$. We split the integral on the right hand side of (\ref{fracint}) into integrals over the regions
$0 \leq t \leq 1$ and $1 < t \leq \infty$. Call these integrals $J_{\alpha}f(x)$ and $K_{\alpha}f(x)$, respectively so that $I_{\alpha}f(x) = J_{\alpha}f(x) + K_{\alpha}f(x)$.
Now $$|J_{\alpha}f(x)| \leq \frac{1}{\Gamma(\alpha/2)}\int_{0}^{1}t^{\alpha/2-1}f^{*}(x)dt = \frac{2}{\alpha}\frac{1}{\Gamma(\alpha/2)}f^{*}(x) < \infty,$$ by finiteness of $f^{*}$.
Furthermore by (\ref{ass1}) (with $p=1$),
$$ |K_{\alpha}f(x)| \leq C_{1}\frac{||f||_{1}}{\Gamma(\alpha/2)}\int_{1}^{\infty}t^{\frac{1}{2}(\alpha - n) -1}dt =
\frac{2||f||_{1}}{(n - \alpha)\Gamma(\alpha/2)} < \infty,$$
and the result follows.
\end{proof}


The next result is stated in \cite{Var} p. 243, equation (0.11). We give a precise proof for the reader's convenience. Let $-A$ be the (self-adjoint) infinitesimal generator of the semigroup $(T_{t}, t \geq 0)$ and assume that $A$ is a positive operator in $L^{2}(S)$. For each $\gamma \in \R$, we can construct the self-adjoint operator $A^{\gamma}$ in $L^{2}(S)$ by functional calculus, and we denote its domain in $L^{2}(S)$ by $\Dom(A^{\gamma})$.

\begin{theorem} \label{negfrac} For all $f \in \Dom(A^{-\frac{\alpha}{2}}) \cap L^{1}(S)$,
$$ I_{\alpha}(f) = A^{-\frac{\alpha}{2}}f,$$ in the sense of linear operators acting on $L^{2}(S)$
\end{theorem}

\begin{proof}  We use the spectral theorem to write $T_{t} = \int_{0}^{\infty}e^{-t\lambda}P(d\lambda)$ for all $t \geq 0$ where $P(\cdot)$ is the projection-valued measure associated to $A$. For all $f \in \Dom(A^{-\frac{\alpha}{2}}), g \in L^{2}(S)$ we have, using Fubini's theorem
\begin{eqnarray} \la I_{\alpha}(f), g \ra & = & \frac{1}{\Gamma(\alpha/2)} \int_0^\infty \int_0^\infty
  t^{\alpha/2-1} e^{- \lambda t} \la P(d\lambda) f, g \ra dt \\
  & = & \frac{1}{\Gamma(\alpha/2)}\left(\int_{0}^{\infty}t^{\alpha/2-1} e^{-t} dt \right)\left(\int_{0}^{\infty}\frac{1}{\lambda^{\frac{\alpha}{2}}}\la P(d\lambda) f, g \ra\right)\nonumber\\
  & = & \la A^{-\frac{\alpha}{2}}f, g \ra \nonumber
  \end{eqnarray}
\end{proof}





\subsection{On Varopoulos' theorem}

The next result is essentially  Theorem 3 in \cite{Var} (see also section II.2 of \cite{VSC}, Corollary 2.4.3 in \cite{Davhk} p.77 and Theorem 4.1 in \cite{CoMe}).  Our proof will follow  the argument in  \cite{Var} (see also \cite{Hed} for a similar approach in the classical case). 
Our assumption that the semigroup is self-adjoint means that the proof is much simpler than in \cite{Var} and we are able to work with $L^{p}$ and $L^{q}$ rather than the corresponding Hardy spaces.

\begin{theorem}\label{HLSth} [Hardy--Littlewood--Sobolev]\label{HLS} Suppose the semigroup $T_t$ has dimension $n$.
  Let $0 < \alpha < n$, $1< p < \frac{n}{\alpha}$ and set $\frac{1}{q} = \frac{1}{p} -
  \frac{\alpha}{n}$.  Then there exists $C_{p,n,\alpha} > 0$ so that for all $f \in L^{p}(S)$,
  \begin{equation}\label{HLS1} ||I_\alpha (f)||_q \leq
  C_{p,n,\alpha} ||f||_p   .
  \end{equation}
\end{theorem}

%

\begin{proof}  Let $\delta>0$ to be chosen later.  Let $x \in S$ be arbitrary and choose $f \in L^{1}(S) \cap L^{p}(S)$ with $f \neq 0$. As in the proof of Lemma \ref{frint1} we split
$I_{\alpha}f(x) = J_{\alpha}f(x) + K_{\alpha}f(x)$ where the integrals on the right hand side range from $1$ to $\delta$ and $\delta$ to $\infty$ (respectively).
Again arguing as in the proof of Lemma \ref{frint1}, we find that
$$ |J_{\alpha}f(x)| \leq \frac{2}{\alpha}\frac{1}{\Gamma(\alpha/2)}f^{*}(x)\delta^{\frac{\alpha}{2}},$$
Now using (\ref{ass1}) we obtain
\bean |K_{\alpha}f(x)| & \leq & C_{p,n,\alpha}\int_{\delta}^{\infty}t^{\frac{\alpha}{2}-\frac{n}{2p}-1}||f||_{p}\\
& \leq & C_{p,n,\alpha}\delta^{\frac{\alpha}{2}-\frac{n}{2p}}||f||_{p},\eean
so that $$|I_{\alpha}f(x)| \leq C_{p,n,\alpha}(f^{*}(x)\delta^{\frac{\alpha}{2}} + \delta^{\frac{\alpha}{2}-\frac{n}{2p}}||f||_{p}).$$
Picking
  $$ \delta = \left( \frac{||f||_p}{f^*(x)} \right)^{2p/n} $$\
  to minimize the right hand side  gives
  \begin{equation}
    |I_{\alpha}f(x)| \leq C_{p,n, \alpha} \left(f^*(x)\right)^{1-\alpha
      p/n} ||f||_p^{\alpha p/n} = C_{p,n, \alpha} \left(f^*(x)\right)^{p/q}
      ||f||_p^{\alpha p/n}.
  \end{equation}\label{majorization}
  Thus for $1<p< \frac{n}{\alpha}$ and using (\ref{ass2}), \begin{eqnarray*}
    ||I_{\alpha}f||^{q}_{q} & \leq & C_{p,n,
        \alpha} ||f||_p^{\alpha pq/n}||f^{*}||_{p}^{p} \\
    &\leq& C_{n, p, \alpha} ||f||_{p}^{p\left(1 + \frac{\alpha q}{n}\right)}\\ & = & C_{n, p, \alpha}||f||_{p}^{q},
  \end{eqnarray*}
  and the required result follows by density.
\end{proof}

We now show how to obtain a Sobolev-type inequality as a corollary to Theorem \ref{HLS}.

\begin{cor} \label{Sob}
For all $1 < p < n, f \in \Dom(A^{\frac{1}{2}}) \cap L^{1}(S)$, if $A^{\frac{1}{2}}f \in L^{p}(S)$ then $f \in L^{\frac{np}{n-p}}(S)$ and
$$ ||f||_{\frac{np}{n-p}} \leq C_{n,p,1}||A^{\frac{1}{2}}f||_{p}.$$
\end{cor}

\begin{proof} Take $\alpha = 1$ so that so that $q = \frac{np}{n-p}$. Applying Theorem \ref{negfrac} within Theorem \ref{HLS} yields
$ ||A^{-\frac{1}{2}}f||_q \leq C_{n,p,1} ||f||_p $ and so on replacing $f$ with $A^{\frac{1}{2}}f$ we find that $||f||_q \leq
  C_{n,p,\alpha} ||A^{\frac{\alpha}{2}}f||_p$ as required.
  \end{proof}

\begin{remark} The domain condition in Corollary \ref{Sob} may seem somewhat strange, but in most cases of interest the operator $A$ and the space $S$ will be such that $\mbox{Dom}(A)^{\frac{1}{2}} \cap L^{1}(S)$ contains a rich set of vectors such as Schwartz space (in $\R^{d}$) or the smooth functions of compact support (on a manifold). In practice, we would only apply the inequality to vectors in that set.
\end{remark}

Note that in the case where $n > 2$ and $p=2$ in Corollary \ref{Sob} we have
$$||f||_{\frac{2n}{n-2}} \leq C_{n,2,1}{\calE}(f),$$
where  ${\calE}(f): = \la Af, f \ra$ is a Dirichlet form.
~If $S$ is a complete Riemannian manifold with bounded geometry (that satisfies our assumptions, see below) and $-A$ is the Laplacian $\Delta$, then we have $n=d$, the dimension of the manifold,  and the Sobolev inequality of Corollary \ref{Sob} takes a more familiar form (c.f. \cite{SC}).

\section{Subordination for Heat Kernels in Euclidean Space}\label{subordination}


In this section, we give examples on both Euclidean spaces and manifolds of non-Gaussian kernels that yield $(n,p)$-ultracontractive semigroups. In each case these semigroups are generated by fractional powers of the Laplacian and are obtained by the technique of subordination.

\subsection{Review of Subordination on Euclidean Space}

For each $\sigma,t > 0$, let $k_{t}^{(\sigma)}: \R^{d} \times \R^{d} \rightarrow (0, \infty)$ denote the heat kernel, i.e.
\begin{equation}\label{heaR^d}
k_{t}^{(\sigma)}(x,y) = \ds\frac{1}{{(2\pi \sigma^{2} t)}^{d/2}}\exp\left\{\frac{|x-y|^{2})}{2\sigma^{2}t}\right\},
\end{equation}
for each $x,y \in \R^{d}$. Then $k^{(\sigma)} \in C^{\infty}((0, \infty) \times \R^{d} \times \R^{d})$ is the fundamental solution of the heat equation: $$\frac{\partial u}{\partial t} = \frac{\sigma^{2}}{2}\Delta u(t)$$ (where the Laplacian $\Delta$ acts on the first spatial variable in $k$). We will only be interested in two values of $\sigma$ in this paper; in this section we use $\sigma = \sqrt{2}$, which is the standard heat kernel of analysis, and for the rest of the paper, $\sigma = 1$ which is the heat kernel of standard Brownian motion.  To simplify notation we will write $\kappa_{t}:=k_{t}^{(\sqrt{2})}$ and $k_{t}: = k_{t}^{(1)}$ for all $t > 0$.

If $u(0) = f \in L^{p}(\Rd) (1 \leq p < \infty)$ then $u(t) = T_{t}f$ for all $t \geq 0$ where $(T_{t}, t \geq 0)$ is the {\it (standard) heat semigroup} defined by $T_{t}f(x) = \int_{\Rd}f(y)\kappa_{t}(x,y)dy$ for $t > 0, x \in \R^{d}$, with $T_{0} = I$.

Now let $0 < \beta < 1$ and for each $t > 0$, let $\gamma_{t}^{\beta}$ be the density of the $\beta$-stable subordinator which is defined uniquely via its Laplace transform by
$$ \int_{0}^{\infty}e^{-ys}\gamma_{t}^{\beta}(s)ds = e^{-ty^{\beta}},$$
for all $y > 0$.
Consider the fractional partial differential equation:
$$ \frac{\partial u}{\partial t} = -(-\Delta)^{\beta} u(t),$$
where for $f \in C_{c}^{\infty}(\Rd)$,
$$ -(-\Delta)^{\beta}f(x) = K_{\beta, d}\int_{\Rd}(f(x+y) - f(x) -y^{i}\partial_{i}f(x){\bf 1}_{|y| < 1})\frac{1}{y^{d + 2\beta}}dy,$$
where $K_{\beta,d} = 2^{\beta}\pi^{-d/2}\Gamma((d+2\beta)/2)\Gamma(1 - \beta)^{-1}$.
It is well known (see e.g. \cite{Appbk}, \cite{SSV}) that this equation has a fundamental solution $q^{\beta}$ which is obtained from the heat kernel by the method of subordination in the sense of Bochner, i.e. for all $t > 0$,
\begin{equation} \label{sub1}
q_{t}^{\beta}(x,y) = \int_{0}^{\infty}\kappa_{s}(x,y)\gamma_{t}^{\beta}(s)ds.
\end{equation}
It follows from the work of \cite{BG} that

\begin{equation} \label{stabest1}
q_{t}^{\beta}(x,y) \asymp C\left(t^{-\frac{d}{2\beta}} \wedge t|x - y|^{-d - 2\beta}\right)
\end{equation}
and as pointed out in (\cite{GHL}), this is equivalent to the estimates

\begin{equation} \label{stabest2}
q_{t}^{\beta}(x,y) \asymp \frac{C}{t^{\frac{d}{2\beta}}}\left(1 + \frac{|x-y|}{t^{\frac{1}{2\beta}}}\right)^{-(d + 2\beta)}.
\end{equation}
Hence, these stable semigroups have dimension $d/\beta$ in the sense of Varopoulos.

\subsection{Stable-Type Transition Kernel on Manifolds}

Much of the structure that we have just described passes over to the case where Euclidean space $\R^{d}$ is replaced by a suitable manifold. To be precise, let $M$ be a complete Riemannian manifold  of dimension $d$ having non-negative Ricci curvature. Let $\Delta$ be the Laplace-Beltrami operator and $\mu$ be the Riemannian volume measure. Then the heat equation: $\frac{\partial u}{\partial t} = \Delta u(t)$ again has a fundamental solution $p \in C^{\infty}((0, \infty) \times M \times M)$ which we again call the {\it heat kernel}. Although there is no precise formula for $p$ we have the heat kernel bounds of Li and Yau \cite{LY},
for all $t > 0, x,y \in M$:
\begin{equation} \label{hk}
p_{t}(x,y) \asymp \frac{C}{V(x,\sqrt{t})}\exp\left(-\frac{\rho(x,y)^{2}}{ct}\right)
\end{equation}
where $\rho$ is the Riemannian metric and for $r > 0, V(x,r)$ is the volume of the ball of radius $r$ centred on $x$. It is well known that for all $x \in M$,
$$ V(x,r) \leq v(d)r^{d}, $$
where $v(d)$ is the volume of the unit ball in $\R^{d}$ (see e.g. \cite{BC}).
We make the following assumption:


\begin{assumption}\label{assump2} There exists $c_{1} > 0$ so that for all $x \in M, V(x,r) \geq c_{1} r^{d}. $
\end{assumption}


Note that as pointed out in \cite[p.~255]{Var}, Assumption \ref{assump2} is equivalent to the following variant on the classical isoperimetric inequality:
$$ \int_{M}|f(x)|^{\frac{2d}{d-2}}\mu(dx) \leq c_{2}\left(\int_{M}|\nabla f(x)|^{2}\mu(dx)\right)^{\frac{1}{2}},$$
where $c_{2} > 0$, for all $f \in C_{c}^{\infty}(M)$.

We thus have that $V(x,r) \asymp r^{d}$. Now let us again consider the fractional partial differential equation $ \frac{\partial u}{\partial t} = -(-\Delta)^{\beta} u(t),$ on $M$ where $0 < \beta < 1$. Just as in the Euclidean space case, the equation has a fundamental solution $\phi^{\beta}$ which is given by subordination, i.e. for all $t > 0, x,y \in M$:
\begin{equation} \label{sub2}
\phi_{t}^{\beta}(x,y) = \int_{0}^{\infty}p_{s}(x,y)\gamma_{t}^{\beta}(s)ds.
\end{equation}
We can now generalise the estimates (\ref{stabest2}):
\begin{theorem} \label{mainest} If Assumption \ref{assump2} holds then for all $t > 0, x,y \in M$
$$\phi_{t}^{\beta}(x,y) \asymp \frac{C}{t^{\frac{d}{2\beta}}}\left(1 + \frac{\rho(x,y)}{t^{\frac{1}{2\beta}}}\right)^{-(d + 2\beta)}.$$
\end{theorem}

\begin{proof} We apply subordination so using (\ref{sub2}), (\ref{hk}) and monotonicity, we have
$$\phi_{t}^{\beta}(x,y) \asymp C\int_{0}^{\infty}\frac{1}{s^{\frac{d}{2}}}\exp\left(-\frac{\rho(x,y)^{2}}{cs}\right)\gamma_{t}^{\beta}(s)ds.$$
We fix $x,y \in M$ and write $\lambda = \rho(x,y)$. Now make a change of variable $s = \frac{4 u}{c}$ and use the scaling property (see e.g. \cite{Appbk}, p.51)
$$ \gamma_{t}^{\beta}(b^{-\frac{1}{\beta}}u) = b^{\frac{1}{\beta}}\gamma_{bt}^{\beta}(u),$$
for all $u > 0$ where $b = \left(\frac{c}{4}\right)^{\beta}$.
Then we obtain
\bean \phi_{t}^{\beta}(x,y) & \asymp & C \int_{0}^{\infty}\kappa_{u}(0,\lambda)\gamma_{bt}^{\beta}(u)du\\
& = & q_{bt}^{\beta}(0,\lambda), \eean by (\ref{sub1}) and the result follows by using (\ref{stabest2}).
\end{proof}


%
%

\section{Fractional Integrals and Martingale transforms on $\R^d$}

In this section we  give a formula for $I_{\alpha}(f)$ as a martingale transform in the case of $\R^d$ and use this to give another proof of Theorem \ref{HLSth} based on martingale inequalities.   Here our semigroup is defined by
$$T_{t}f(x) = \int_{\R^d}f(y)k_{t}(x, y)dy$$
where we emphasise that from now on,
$$
k_{t}(x, y)=k_t(x-y)=\frac{1}{(2\pi t)^{d/2}} e^{\frac{|x-y|^2}{2t}}.
$$
Thus in the language our Assumption \ref{assumption1}, this semigroup has dimension $d$, the same as the space where it is defined.   As before,
\begin{equation}\label{fracinteuclidean}
  I_\alpha (f)(x) = \frac{1}{\Gamma(\alpha/2)} \int_0^\infty
  t^{\alpha/2-1} T_{t}f(x)\, dt= f * R_{\alpha,d}
\end{equation}
where $*$ is convolution of functions and for all $x \in \R$, $$R_{\alpha, d}(x) = \ds\frac{\Gamma\left(\frac{d-\alpha}{2}\right)}{\Gamma\left(\frac{\alpha}{2}\right)2^{\frac{\alpha}{2}}\pi^{\frac{d}{2}}|x|^{d-\alpha}},$$
is the Riesz kernel (see e.g. \cite{FOT}, p.43).
The last line is a simple computation once the explicit expression for $T_tf$ as a convolution of $f$ with $k_t$ is substituted in the formula for $I_{\alpha}$. Note that (up to a multiplicative constant) we recapture the classical Riemann-Liouville fractional integral when $d=1$. The operator $I_{\alpha}$ is sometimes called the Riesz potential (see e.g. \cite{Hed}).

Our first goal is to give a formula for $I_{\alpha}f$ as the conditional expectation of a stochastic integral.  For this we follow the exact same approach as the one presented in \cite{BanMen} which represents the Beurling-Ahlfors operator as the projection of martingales with respect to space-time Brownian motion.  For further examples of this technique, see \cite{AppBan} and \cite{Ban} and the many references in these papers.

\subsection{Stochastic Integral representation for $I_{\alpha}$}

Let $B_t$ be Brownian motion in $\R^d$.  For $f \in {\calS}(\R^{d})$ and fixed $a>0$, which we think of as being very large,  we consider the pair of martingales up to time $a$ given by
\begin{equation}
M_f^a(t)=\int_0^{a\wedge t}\nabla(T_{a-{s}}f)(B_s)\cdot dB_s
\end{equation}
and
\begin{equation}
M_f^{a, \alpha}(t)=\int_0^{a\wedge t}(a-s)^{\alpha/2}\nabla(T_{a-{s}}f)(B_s)\cdot dB_s.
\end{equation}
We note that by the It\^o formula,
\begin{equation} \label{use}
T_{a-t}f(B_t)=T_af(B_0)+M_f^a(t), \,\,\, 0<t<a,
\end{equation}

Standard calculations yield that the quadratic variation of these martingales are $$[M_f^{a}](t) = \int_{0}^{a \wedge t}|\nabla(T_{a-{s}}f)(B_s)|^{2}ds$$ and
$$[M_f^{a, \alpha}](t) = \int_{0}^{a\wedge t}(a-s)^{\alpha}|\nabla(T_{a-{s}}f)(B_s)|^{2}ds.$$ Since for any $0<s<t<a$, $(a-s)^{\alpha}<a^{\alpha}$, we conclude that $$ [M_f^{a, \alpha}](t) \leq a^{\alpha}[M_f^{a}](t),$$ for all $t \geq 0$. It follows that the continuous martingale $M_f^{a, \alpha}(t)$ is differentially subordinate to $a^{\alpha}M_f^a(t)$ ( see \cite{Ban} for details) and so
for any $1<p<\infty$ we have, by the celebrated Burkholder's inequalities, that
\begin{equation}
\|M_f^{a, \alpha}(a)\|_p\leq a^{\alpha}(p^*-1)\|M_f^a(a\|_p,\,\,\, 1<p<\infty,
\end{equation}
where $$p^{*}=\max\left\{p,\frac{p}{p-1}\right\}.$$
We note, however, that while this holds for all $1<p<\infty$, the bound  depends on $a$ and this does not aid our quest to obtain a probabilistic proof of the Hardy-Littlewood-Sobolev inequality.  What we seek is an inequality of this type, but with a bound independent of $a$, and this requires placing some restrictions on $p$, as in the Hardy-Littlewood-Sobolev inequality.

Let us first determine the nature of the transformation giving rise to $M_f^{a, \alpha}(t)$.  Set  $t=a$ in (\ref{use}) to obtain
\begin{equation}
f(B_a)=T_af(B_0)+\int_0^{a}\nabla(T_{a-{s}}f)(B_s)\cdot dB_s.
\end{equation}
If $g \in {\calS}(\R^{d})$, we have
\begin{eqnarray}\label{product}
&&g(B_a)\int_0^{a}(a-s)^{\alpha/2}\nabla(T_{a-{s}}f)(B_s)\cdot dB_s\\&=&T_ag(B_0)\int_0^{a}(a-s)^{\alpha/2}\nabla(T_{a-{s}}f)(B_s)\cdot dB_s\nonumber\\
&+&\left(\int_0^{a}\nabla(T_{a-{s}}g)(B_s)\cdot dB_s\right)\left(\int_0^{a}(a-s)^{\alpha/2}\nabla(T_{a-{s}}f)(B_s)\cdot dB_s\right).\nonumber
\end{eqnarray}
Observe further that the expectation of the first term is zero. That is,
\begin{eqnarray*}
&&\mathbb{E}\left(T_ag(B_0)\int_0^{a}(a-s)^{\alpha/2}\nabla(T_{a-{s}}f)(B_s)\cdot dB_s\right)\\
&=&\int_{\R^d}\mathbb{E}_{x}\left(T_ag(B_0)\int_0^{a}(a-s)^{\alpha/2}\nabla(T_{a-{s}}f)(B_s)\cdot dB_s\right)dx\\
&=&\int_{\R^d}T_ag(x)\mathbb{E}_{x}\left(\int_0^{a}(a-s)^{\alpha/2}\nabla(T_{a-{s}}f)(B_s)\cdot dB_s\right)dx\\
&=& 0
\end{eqnarray*}
where here and henceforth, $\E$ denotes the expectation of the Brownian motion with initial distribution the Lebesgue measure.  (See \cite{BanMen} for more on this construction.) Thus by It\^{o'}s isometry,
\begin{eqnarray}\label{ItoIso}
&&\mathbb{E}\left(g(B_a)\int_0^{a}(a-s)^{\alpha/2}\nabla(T_{a-{s}}f)(B_s)\cdot dB_s\right)\\
&=&\mathbb{E}\left(\int_0^{a}\nabla(T_{a-{s}}g)(B_s)\cdot dB_s\right)\left(\int_0^{a}(a-s)^{\alpha/2}\nabla(T_{a-{s}}f)(B_s)\cdot dB_s\right)\nonumber\\
&=&\mathbb{E} \left(\int_0^{a}(a-s)^{\alpha/2}\nabla(T_{a-{s}}f)(B_s)\cdot\nabla (T_{a-{s}}g)(B_s)ds\right)\nonumber
\end{eqnarray}

For $f$, $a$ and $\alpha$ as above, we define for all $x \in \R^{d}$,
\begin{equation}
\mathcal{S}^{a, \alpha} f (x) =\mathbb{E} \left(\int_0^{a}(a-s)^{\alpha/2}\nabla(T_{a-{s}}f)(B_s)\cdot dB_s  \mid B_a=x \right).
\end{equation}

\begin{theorem}\label{stochasticrep}
For all $f \in {\calS}(\R^{d}), x \in \R^{d}$
\begin{equation} \label{uptoa}
\mathcal{S}^{a, \alpha} f(x) =-\int_0^a s^{\alpha/2}T_s{(\Delta T_s f)}(x)ds
\end{equation}
and  as $a\to\infty$,
\begin{equation}\label{uptoinfinity}
\mathcal{S}^{a, \alpha} f(x) \to I_{\alpha}(f)(x).
\end{equation}
almost everywhere.
\end{theorem}


\begin{proof}  We first observe that for  $f\in {\calS}(\R^{d})$ we have
\begin{eqnarray}\label{infiniteprobability}
\mathbb{E} (f(B_a))=\int_{\R^d}\mathbb{E} _{x}(f(B_a)) dx&=&\int_{\R^d}\left(\int_{R^d}f(\tilde{x})p_a(x-\tilde{x})d\tilde x\right)dx\nonumber\\
&=&\int_{\R^d} f(\tilde{x})d\tilde{x}.
\end{eqnarray}
Let $g \in {\calS}(\R^{d})$.  Then, by the above calculations, integration by parts and self-adjointness of the semigroup, we have
\begin{eqnarray*}
\int_{\R^d} \mathcal{S}^{a, \alpha} f(x) g(x) dx&=&\int_{\R^d}\mathbb{E} \left(\int_0^{a}(a-s)^{\alpha/2}\nabla(T_{a-{s}}f)(B_s)\cdot dB_s  \mid B_a=x \right) g(x) dx\\
&=&\mathbb{E} \left(\mathbb{E}\left(\int_0^{a}(a-s)^{\alpha/2}\nabla(T_{a-{s}}f)(B_s)\cdot dB_s  \mid B_a \right) g(B_a) \right)\\
&=&\mathbb{E} \left(\mathbb{E}\left(g(B_a)\int_0^{a}(a-s)^{\alpha/2}\nabla(T_{a-{s}}f)(B_s)\cdot dB_s \right) \left|B_a  \right. \right)\\
&=&\mathbb{E} \left(g(B_a) \int_0^{a}(a-s)^{\alpha/2}\nabla(T_{a-{s}}f)(B_s)\cdot dB_s \right) \\
&=&\mathbb{E} \left(\int_0^{a}(a-s)^{\alpha/2}\nabla(T_{a-{s}}f)(B_s)\cdot\nabla (T_{a-{s}}g)(B_s)ds\right)\\
&=& \int_0^{a}\left\{s^{\alpha/2}\int_{R^d}\nabla(T_{{s}}f)(x)\cdot\nabla (T_{{s}}g)(x)dx\right\} ds\\
&=&-\int_0^{a}\left\{s^{\alpha/2}\int_{R^d}\Delta(T_{{s}}f)(x)(T_{{s}}g)(x)dx\right\} ds\\
&=&-\int_0^{a}\left\{s^{\alpha/2}\int_{R^d}T_s\left(\Delta(T_{{s}}f\right))(x)g(x)dx\right\} ds\\
&=&-\int_{R^d}\left\{ \int_0^{a}s^{\alpha/2}\,T_s\left(\Delta(T_{{s}}f\right))(x) ds\right\}g(x)dx.
\end{eqnarray*}
This completes the proof of \eqref{uptoa}.

Now recall that $\frac{d}{dt}T_{t}f = \Delta T_{t}f$. Write $u(t, \cdot) = T_{t}f$, then $\frac{\partial}{\partial t}u(t, \cdot) = \Delta u(t, \cdot)$ and so
$$ \frac{\partial}{\partial t}u(2t, \cdot) = 2 u^{\prime}(2t, \cdot) = 2 \Delta u(2t, \cdot).$$
This gives that
$$ \Delta T_{2s}f = \frac{1}{2} \frac{d}{ds}T_{2s}f$$
and hence
\begin{eqnarray*}
 \mathcal{S}^{a, \alpha} f(x)&=&-\int_0^{a}s^{\alpha/2}\,\Delta(T_{2s})f(x) ds\\
 &=&-{\frac{1}{2}}\int_0^{a}s^{\alpha/2}\,\frac{dT_{2s}f}{ds}(x)ds\\
 &=& -{\frac{1}{2}}a^{\alpha/2}T_{2a}f(x)+\frac{\alpha}{4}\int_0^{a}s^{\alpha/2-1} T_{2s}f(x) ds.
 \end{eqnarray*}
 Since $|T_{2a}f(x)|\leq \frac{C}{a^{d/2}}\|f\|_1$ and $0<\alpha<d$, $a\to\infty$, the right hand side of the previous equality goes to
 $$
 \frac{\alpha}{4}\int_0^{\infty}s^{\alpha/2-1} T_{2s}f(x) ds= 2^{-\frac{\alpha + 4}{2}}\alpha \Gamma\left(\alpha/2\right) I_\alpha f(x)
 $$
 and this proves \eqref{uptoinfinity}.
 \end{proof}

 \begin{remark}
 This derivation works in the setting of the manifolds studied in \cite{BanBau}; see the proof of Lemma 3.2 in that paper.  Hence it will also work on Lie groups as in \cite{AppBan}. These directions will not be explored in this paper.
 \end{remark}


 Our goal is now to use the formula in \eqref{uptoa} to give a proof of  Hardy-Littlewood-Sobolev inequality in Theorem \ref{HLSth} using martingale inequalities.  We begin with the following simple proposition which follows just by  differentiation of the Gaussian kernel.  We give its proof for completeness. We use the notation $k_{t}(x): = k_{t}(x, 0)$ for each $x \in \R^{d}, t > 0$.


\begin{prop}\label{hkder}
 For all $x \in \R^{d}, t > 0$,
 \begin{equation}\label{gradestimate}
  \nabla_x k_t(x) \leq 2^{\frac{d+4}{2}}\frac{1}{\sqrt{t}} k_{2t}(x).
  \end{equation}
 \end{prop}

\begin{proof} We start by observing that
  \begin{equation*}
  \nabla_x k_t(x)=-\left(\frac{x_1}{t}, \cdots \frac{x_d}{t}\right)k_t(x)
  \end{equation*}
 so that
 \begin{eqnarray*}
 |\nabla_x k_t(x)|&\leq& \frac{1}{\sqrt{t}}\sqrt{\frac{|x|^2}{t}}k_t(x)\nonumber\\
 &=&\frac{1}{\sqrt{t}}\sqrt{\frac{|x|^2}{t}}\frac{1}{(2\pi t)^{d/2}}e^{-\frac{|x|^2}{2t}}\nonumber
 \end{eqnarray*}

 We now claim that the right hand side is dominated by
 $2^{\frac{d+4}{2}}\frac{1}{\sqrt{t}} k_{2t}(x)$.
To see this, observe that if $\sqrt{\frac{|x|^2}{t}}\leq 1$, then the right hand side is dominated by
 $
 \frac{1}{\sqrt{t}}\frac{1}{(2\pi t)^{d/2}}e^{-\frac{|x|^2}{2t}}.
 $
 If
 $a=\sqrt{\frac{|x|^2}{t}}>1$, then $a<a^2=4(a/2)^2\leq 4e^{\frac{a^2}{4}}$ and the right hand side is dominated by
 $$
 4\frac{1}{\sqrt{t}}\frac{1}{(2\pi t)^{d/2}}e^{(-\frac{|x|^2}{2t}+\frac{|x|^2}{4t})}=4\frac{1}{\sqrt{t}}\frac{1}{(2\pi t)^{d/2}}e^{-\frac{|x|^2}{4t}}.
 $$
Since $e^{-\frac{|x|^2}{2t}}\leq e^{-\frac{|x|^2}{4t}}$, we see that in either case, the right hand side of \eqref{gradestimate} is dominated by
$$
4\frac{1}{\sqrt{t}}\frac{1}{(2\pi t)^{d/2}}e^{-\frac{|x|^2}{4t}}=2^{\frac{d+4}{2}}\frac{1}{\sqrt{t}} k_{2t}(x)
$$
and this completes  the proof.
\end{proof}
 \begin{remark} The estimate (\ref{gradestimate}) which is the key to the calculations below holds more widely on manifolds, see \cite{LY}, \cite{AusCouDuoHof} for much more on these type of bounds on heat kernels.
 \end{remark}

 We now fix $0 < \alpha < d$ and set $\frac{1}{q} = \frac{1}{p} -
  \frac{\alpha}{d}$, for $1<p < \infty$,  and as always work with functions in ${\calS}(\R^{d})$.   We assume that $a$ is very large but fixed for now.  By the classical Burkholder-Gundy inequalities there is a constant $C_q$ independent of $a$ so that for all $t \geq a$
\begin{eqnarray}\label{Bur-Gun}
||M_{f}^{a, \alpha}(t)||_{q} & = &\Big\|\int_0^{a}(a-s)^{\alpha/2}\nabla(T_{a-{s}}f)(B_s)\cdot dB_s\Big\|_q\nonumber\\
&\leq&
C_q\Big\|\left(\int_0^{a}(a-s)^{\alpha}|\nabla(T_{a-{s}}f)(B_s)|^2\, ds\right)^{1/2}\Big\|_q,
\end{eqnarray}
where, as in (\ref{infiniteprobability}), for all $ 1 < p < \infty, h \in L^{p}(\R^{d}), t \geq 0$,
$$ ||h(B_{t})||_{p} = (\E(|h(B_{t})|^{p})^{\frac{1}{p}} = \left(\int_{\R^{d}}|h(x)|^{p}dx\right)^{\frac{1}{p}} = ||h||_{p}.$$

\begin{lemma} \label{newdelta}
Let $\delta > 0$ be arbitrary. Then there exists $C_{1}, C_{2} \geq 0$ so that
\begin{equation} \label{newdelta1}
\int_0^{a}(a-s)^{\alpha}|\nabla(T_{a-s}f)(B_s)|^{2}ds \leq C_{1}\left(\sup_{0<s<a}|(T_{2(a-s)}|f|)(B_{s})|\right)^2\delta^{\alpha}
+ C_{2}||f||_{p}^{2}\delta^{\alpha - n/p}.
\end{equation}
\end{lemma}

\begin{proof} There are two cases to consider:


{\it Case 1:} $\delta > a$. Using the estimate of Proposition \ref{hkder} for the derivative of the heat kernel, for some $c_{1} > 0$ depending only on $d$,
\bean \int_0^{a}(a-s)^{\alpha}|\nabla(T_{a-s}f)(B_s)|^{2}ds & \leq & c_{1} \int_0^{a}(a-s)^{\alpha} \frac{1}{a-s}|(T_{2(a-s)}|f|)(B_s)|^{2}ds\\
& \leq & c_{1}\sup_{0<s<a}|(T_{2(a-s)}|f|)(B_{s})|^{2}\int_{0}^{a}(a-s)^{\alpha -1}ds\\
& \leq & c_{1}\sup_{0<s<a}|(T_{2(a-s)}|f|)(B_{s})|^{2}\int_{0}^{\delta}s^{\alpha -1}ds\\
& \leq & \frac{c_{1}}{\alpha}\sup_{0<s<a}|(T_{2(a-s)}|f|)(B_{s})|^{2}\delta^{\alpha}, \eean
and the estimate (\ref{newdelta1}) holds with $C_{2} = 0$.


{\it Case 2:} $\delta < a$. Here, as before, write
\begin{eqnarray*}
\int_0^{a}(a-s)^{\alpha}|\nabla(T_{a-s}f)(B_s)|^2\, ds&=&\int_0^{a}s^{\alpha}|\nabla(T_{s}|f|)(B_{a-s})|^2\, ds\\
&=&\int_0^{\delta}s^{\alpha}|\nabla(T_{s}|f|)(B_{a-s})|^2\, ds\\
&+&\int_{\delta}^{a}s^{\alpha}|\nabla(T_{s}|f|)(B_{a-s})|^2\, ds\\
&=&I +II.
\end{eqnarray*}
Note that by Proposition \ref{hkder} again for some $c_{2} > 0$ depending only on $d$,
\begin{eqnarray}\label{canweusedoob}
I&\leq& c_{2}\int_0^{\delta}s^{\alpha-1}|(T_{2s}|f|)(B_{a-s})|^2ds\nonumber\\
& \leq & c_{2} \sup_{0 \leq s \leq \delta}|(T_{2s}|f|)(B_{a-s})|^2\int_0^{\delta}s^{\alpha-1}ds\\
& \leq & \frac{c_{2}}{\alpha} \sup_{0 \leq s \leq a}|T_{2(a-s)}|f|)(B_{s})|^2 \delta^{\alpha}\nonumber
\end{eqnarray}

Next we use the estimate of Proposition \ref{hkder} and the assumption (\ref{ass1}) on the $(d,p)$-ultracontractivity of the semigroup to conclude that
$$
|\nabla(T_{s}f)(B_{a-s})|^2\leq \frac{c_3}{s}|(T_{2s}|f|)(B_{a-s})|^{2} \leq \frac{c_4}{s^{d/p + 1}}||f||^{2}_{p}
$$
and therefore,
\begin{eqnarray}\label{ultraworks}
II&\leq& c_{4}||f||_{p}^{2}\int_{\delta}^{a}\frac{1}{s^{d/p+1-\alpha}}\,ds\nonumber\\
&\leq& c_{5}||f||_{p}^{2}\int_{\delta}^{\infty}\frac{1}{s^{d/p+1-\alpha}}\,ds\\
& \leq & c_{6}||f||_{p}^{2}\delta^{\alpha - d/p}\nonumber.
\end{eqnarray}
The result follows.
\end{proof}


Using Lemma \ref{newdelta} we see that
\begin{eqnarray} \label{Deltamin}
\left(\int_0^{a}(a-s)^{\alpha}|\nabla(T_{a-{s}}f)(B_s)|^2\, ds\right)^{1/2}&\leq & C_{\alpha, d}\sup_{0<s<a}|T_{2(a-s)}|f|(B_{s})|\delta^{\alpha/2} \nonumber \\
&+&C_{p,\alpha, d}\|f\|_p\,\delta^{\frac{\alpha}{2}-\frac{d}{2p}}.
\end{eqnarray}
Minimizing this inequality in $\delta$ as before, we find that
\begin{eqnarray}\label{key1}
  && \left(\int_0^{a}(a-s)^{\alpha}|\nabla(T_{a-s}|f|)(B_s)|^2\, ds\right)^{1/2} \\
  &\leq& C_{p, \alpha, d} \left(\sup_{0<s<a}|(T_{2(a-s)}f)(B_{s})|\right)^{1-\alpha
      p/d} ||f||_p^{\alpha p/d}\nonumber\\
      &= &C_{p, \alpha, d} \left(\sup_{0<s<a}|(T_{2(a-s)}|f|)(B_{s})|\right)^{p/q}
      ||f||_p^{\alpha p/d}, \nonumber
 \end{eqnarray}
where the constant $C_{p, \alpha, d}$ depends only on the parameters indicated.  In particular (and this is important), this constant does not depend on $a$.

\begin{remark}
 We remark that the value of $\delta$ that minimises (\ref{Deltamin}) depends on $\omega \in \Omega$, however the generality of Lemma \ref{newdelta} ensures the validity of this procedure.
 \end{remark}

 \begin{lemma} For $1<p<\infty, f \in S(\R^{d})$, and all $a>0$ there is a constant $C_p$ independent of $a$ such that
 \begin{equation}\label{question}
 \Big\|\left(\sup_{0<s<a}|(T_{2(a-s)}f)(B_{s})|\right)\Big\|_p\leq C_p\, ||f||_{p}
 \end{equation}
 where the norm is taken with respect to the expectation $\E$ as above.
 \end{lemma}

\begin{proof} For all $0 \leq t \leq a$, define $Y_{t}(f):=T_{2(a-t)}f(B_{t})$. We first show that $(Y_{t}(f), 0 \leq t \leq a)$ is a martingale. Define the process $\{M_{t}(f), 0 \leq t \leq a\}$ by $$M_{t}(f):=\int_{0}^{t}\nabla T_{2(a-s)}f(B_{s}) \cdot dB_{s}.$$ Then this process is a local martingale. To see that it is in fact a martingale, its enough to show that it is square-integrable. Using It\^{o}'s isometry, Proposition \ref{hkder} and the $(d,p)$-ultracontractivity assumption (\ref{ass1}), we find that for all $0 \leq t \leq a$,
\bean \E(M_{t}(f)^{2}) & = & \int_{0}^{t}\E(\nabla T_{2(a-s)}f(B_{s})^{2})ds \\
& \leq & \frac{1}{2}\int_{0}^{t}\frac{1}{a-s}\E(T_{4(a-s)}f(B_{s})^{2})ds \\
& \leq & C_{1}||f||_{p}\int_{0}^{t}(a-s)^{-\frac{d}{p} - 1}ds\\
& = & C_{2}||f||_{p}[(a-t)^{-\frac{d}{p}} - a^{-\frac{d}{p}}] < \infty. \eean

 By It\^{o}'s formula
 \bean Y_{t}(f) & = & Y_{0}(f) + M_{t}(f) - \frac{1}{2}\int_{0}^{t}\triangle T_{2(a-s)}f(B_{s})ds \\
 & = & Y_{0}(f) + M_{t}(f) + \frac{1}{2}\int_{0}^{t}\frac{d}{ds} T_{2(a-s)}f(B_{s})ds \\
 & = & T_{2a}f(B_{0}) +  M_{t}(f) + \frac{1}{2} Y_{t}(f) - \frac{1}{2}T_{2a}f(B_{0}), \eean from which we deduce that
 $$ Y_{t}(f) = 2M_{t}(f) + \frac{1}{2}T_{2a}f(B_{0}).$$
Hence  $\{Y_{t}(f), 0 \leq t \leq a\}$ is a martingale.

Note that by (\ref{infiniteprobability}) $E(|f(B_a)|^p)=\|f\|_p^p$.
Using this together with  Doob's maximal inequality we find that
\bean \E\left(\sup_{0<s<a}|T_{a-s}f(B_{s})|\right)^p & =  & \E\left(\sup_{0<s<a}|Y_{s}(f)|^{p}\right)\\
& \leq & \left(\frac{p}{p-1}\right)^{p}\E(|Y_{a}(f)|^{p})\\
& = & \left(\frac{p}{p-1}\right)^{p}\, E|f(B_a)|^p \\
& = & \left(\frac{p}{p-1}\right)^{p}||f||_{p}^{p} \eean
and this gives the desired inequality.
\end{proof}

\begin{cor}\label{key2} For $1 < p < \infty$, \begin{equation}\label{question1}
 \Big\|\left(\sup_{0<s<a}|(T_{2(a-s)}|f|)(B_{s})|\right)\Big\|_p\leq 2C_p||f||_{p}.
 \end{equation}
 \end{cor}

\begin{proof} Let $f_{+} = \max\{f, 0\}$ and $f_{-} = \max\{-f, 0\}$. By the smoothing effect of the semigroup we may apply (\ref{question}) where $f$ is replaced by $f_{+}$ and $f_{-}$ (respectively) and then we have
\begin{eqnarray*}
& &  E\left(\sup_{0<s<a}|(T_{2(a-s)}|f|)(B_{s})|\right)\\ & \leq & E\left(\sup_{0<s<a}|T_{2(a-s)}f_{+})(B_{s})|\right) + E\left(\sup_{0<s<a}|(T_{2(a-s)}f_{-})(B_{s})|\right)\\
 & \leq & C_{p}(||f_{+}||_{p} + ||f_{-}||_{p}) \leq 2C_p||f||_{p},
 \end{eqnarray*}
 which gives the result.
 \end{proof}

We now proceed to show how a probabilistic proof of Theorem \ref{HLSth} for the heat semigroup follows from our constructions. Recall that $n=d$ in this case, fix $0 < \alpha < d$ and set $\frac{1}{q} = \frac{1}{p} - \frac{\alpha}{d}$, for $1<p < \infty$. By Theorem \ref{stochasticrep}, the contraction of the $L^q$ norm by the conditional expectation and the classical Burkholder-Gundy inequalities, there is a constant $C_q$ independent of $a$ so that for all $f \in {\calS}(\Rd)$,

\begin{eqnarray}
\Big\|\mathcal{S}^{a, \alpha} f\Big\|_q
& \leq &\Big\|\int_0^{a}(a-s)^{\alpha/2}\nabla(T_{a-{s}}f)(B_s)\cdot dB_s\Big\|_q\\
&\leq&
C_q\Big\|\left(\int_0^{a}(a-s)^{\alpha}|\nabla(T_{a-{s}}f)(B_s)|^2\, ds\right)^{1/2}\Big\|_q,\nonumber
\end{eqnarray}
 where the norm on the left hand side is on $\R^d$ with respect to the Lebesgue measure and the right hand side is with respect to $\E$.

 By inequality \eqref{key1} and Corollary \ref{key2},
 \begin{eqnarray} \Big\|\left(\int_0^{a}(a-s)^{\alpha}|\nabla(T_{a-{s}}f)(B_s)|^2\, ds\right)^{1/2}\Big\|_q&\leq &
 C_{p, \alpha, d}\|f\|_p^{\frac{p}{q}} \|f\|_p^{\frac{\alpha p}{d}}.\\
 &=&C_{p, \alpha, d}\|f\|_p,\nonumber
 \end{eqnarray}
 Since this bound  does not depend on $a$, letting $a\to\infty$ and applying Fatou's lemma, Theorem \ref{stochasticrep} and the density of ${\calS}(\Rd)$ in $L^{q}(\Rd)$ gives the result.



An alternative stochastic representation can be carried out using the Gundy-Varopoulos \cite{GunVar} construction instead of the space-time Brownian process $(B_t, a-t), 0<t<a$ from \cite{BanMen}.   Such a construction will also work on a manifold. But even more, this construction will work for any semigroup which, in addition to the ultracontractivity property $|T_tf(x)|\leq \frac{C}{t^{n/2}}\|f\|_1$, satisfies the assumptions of Varopoulos \cite{Var1}.  We briefly explain the construction on $\R^d$.  We let $T_t$ be the heat semigroup and construct its Poisson semigroup by subordination with $\beta=1/2$ as in \S\ref{subordination} above.  We denote this semigroup by $P_t$ and to conform to more classical notation, we use $y>0$ in place of $t$.  Hence the semigroup is denoted by $P_y$.  Given $f\in {\calS}(\R^{d})$ we set $u_f(x, y)=P_yf(x)$, $y\geq 0$, $x\in \R^d$, again to conform to the standard notation.   We again fix a large $a>0$ and let $Z_t=(B_t, Y_t)$ be Brownian motion in $\R_{+}^{d+1}$ starting on the hyperplane $(x, a)$ with initial distribution the Lebesgue measure.  That is, we start at each point $(x, a)$ and integrate the initial distribution with respect to $x$.  This gives expectation which we denote by $\E^a$.  If we let
 $$
 \tau_{a}=\inf\{t>0: Y_t=0\},
 $$
 then we see that for any $f\in {\calS}(\R^{d})$,
 $$E^a f(B_{\tau_a})=\int_{\R^d}f(x)dx,$$
just as before.

For $f$, $a$ and $\alpha$ as above, we define
\begin{equation}\label{Gun-Var}
\mathcal{T}^{a, \alpha} f (x) =\mathbb{E}^a \left(\int_0^{\tau_a} Y_{t}^{\alpha}\frac{\partial u_f}{\partial y}(B_t, Y_t)dY_t  \mid B_{\tau_a}=x  \right).
\end{equation}

\begin{theorem}\label{Gundy-Var}
For all $f \in {\calS}(\R^{d})$, as $a \rightarrow \infty$
\begin{equation}
\mathcal{T}^{a, \alpha} f  \to C_{\alpha}I_{\alpha}f,
\end{equation}
for some constant $C_{\alpha}$, in the sense that
\begin{equation}
\int_{\R^d}\mathcal{T}^{a, \alpha} f (x) g(x) dx\to C_{\alpha}\int_{\R^d}I_{\alpha}f(x) g(x) dx,
\end{equation}
fo all $f, g\in {\calS}(\R^{d})$.
 \end{theorem}

 The proof of this Theorem is exactly the same as the proof given in \cite[\S 3.4]{Ban} for the representation of the Riesz transforms and we leave it to the reader. We also refer the reader to  \cite{Var1} where these type of arguments are presented for general semigroups.  In particular, the same argument will work if instead of $\R^d$ we take a manifold $M$ with a Brownian motion $X_t$ and consider the space $M\times (0, \infty)$ with the Brownian motion $(X_t, Y_t)$ where $Y_t$ is a one dimensional  Brownian motion killed the first time it hits $0$.

\vspace{5pt}

\noindent {\it Acknowledgements.} David Applebaum would like to thank both the London Mathematical Society and Purdue University for the financial support which enabled this project to get off the ground during the summer of 2012. Both authors would like to thank Krzysztof Bogdan for inviting them to the 6th international conference on stochastic analysis at B\c edlewo in September 2012 where we able to make much progress.

\end{document}